\documentclass[a4paper, 11pt]{amsart}

\usepackage{amsmath,amssymb,verbatim,graphicx,amsthm,color}
\usepackage[top=1in, bottom=1in, left=1in, right=1in]{geometry}
\usepackage{cite}
\allowdisplaybreaks[1]
\usepackage[sc]{mathpazo}
\usepackage[T1]{fontenc}

\usepackage[backref=page,colorlinks=true]{hyperref} 
\usepackage{ifthen}
\usepackage{tikz}
\usetikzlibrary{decorations.pathreplacing}
\usepackage{tikz-cd}
\renewcommand*{\backref}[1]{}
\renewcommand*{\backrefalt}[4]{\tiny
	\ifcase #1 (\textbf{NOT CITED.})%
	\or    (Cited on page~#2.)%
	\else   (Cited on pages~#2.)%
	\fi}

\newcommand{\vertiii}[1]{{\left\vert\kern-0.25ex\left\vert\kern-0.25ex\left\vert #1 
    \right\vert\kern-0.25ex\right\vert\kern-0.25ex\right\vert}}
\newcommand{\braciii}[1]{{\left[\kern-0.25ex\left[\kern-0.25ex\left[ #1 
		\right]\kern-0.25ex\right]\kern-0.25ex\right]}}
\newcommand{\norm}[1]{\left\lVert#1\right\rVert}

\newcommand{\RR}{\mathbb{R}}
\newcommand{\CC}{\mathbb{C}}
\newcommand{\ZZ}{\mathbb{Z}}

\newcommand{\NN}{\mathbb{N}}
\newcommand{\TT}{\mathbb{T}}

\newcommand{\ind}{\mathbb{1}}

\theoremstyle{plain}

\newtheorem{theorem}{Theorem}[section]

\newtheorem{lemma}[theorem]{Lemma}

\title[Good weight for commuting transformations]{A good universal weight for multiple recurrence averages with commuting transformations in norm}

\author{Idris Assani}
\address{Department of Mathematics, The University of North Carolina at Chapel Hill, 
Chapel Hill, NC 27599}
\email{assani@math.unc.edu}
\urladdr{https://idrisassani.web.unc.edu/} 

\author{Ryo Moore}
\address{Department of Mathematics, Southern University of Science and Technology, Shenzhen, Guangdong, China}
\email{ryomoore@sustech.edu.cn}
\urladdr{http://sites.google.com/view/ryomoore} 

\usepackage{setspace}
\begin{document}
\maketitle
\begin{abstract}
	We will show that the sequences appearing in Bourgain's double recurrence result are good universal weights to the multiple recurrence averages with commuting measure-preserving transformations in norm. This will extend the pointwise converge result of Bourgain, the norm convergence result of Tao, and the authors' previous work on the single measure-preserving transformation.
\end{abstract}

\section{Introduction}
We denote $(X, \mathcal{F}, \mu, T)$ to be a measure-preserving system, where $(X, \mathcal{F}, \mu)$ is a probability measure space. If there are more than one measure transformations, say $T_1, \ldots, T_k$ for some $k \in \NN$, acting on the probability space $(X, \mathcal{F}, \mu)$, we will list out all of the transformations at the end by denoting $(X, \mathcal{F}, \mu, T_1, T_2, \ldots, T_k)$.

This note is a sequel to \cite{AM}, and we refer the readers to that article for more detailed background on this subject.

\subsection{The main theorem}
In this short note, we will prove the following:

\begin{theorem}\label{mainThm}
	Let $(X, \mathcal{F}, \mu, T)$ be an ergodic, invertible, measure-preserving system, and suppose $f_1, f_2 \in L^\infty(\mu)$, and $a, b$ be distinct nonzero integers. Then there exists a set $X' \subset X$ such that $\mu(X') = 1$, and for every $x \in X'$ and any other measure-preserving system with $k$ commuting transformations $(Y, \mathcal{G}, \nu, S_1, S_2, \ldots, S_k)$ and $g_1, g_2, \ldots, g_k \in L^\infty(\nu)$, the averages
	\[ \frac{1}{N} \sum_{n=1}^N f_1(T^{an}x)f_2(T^{bn}x) \prod_{i=1}^k g_i \circ S_i^n  \]
	converge in $L^2(\nu)$.
\end{theorem}
We will provide some context here: The a.e. pointwise convergence result of the averages
\[\frac{1}{N} \sum_{n=1}^N f_1(T^{an}x)f_2(T^{bn}x) \]
was shown by Bourgain \cite{BoDR}, while the $L^2(\nu)$-convergence of the averages
\[ \frac{1}{N} \sum_{n=1}^N \prod_{i=1}^k g_i \circ S_i^n  \]
was originally obtained by Tao \cite{T}. These types of averages are often referred to as multiple recurrence averages (or nonconventional ergodic averages), and they stem from Furstenberg's study on proving the Szemer\'edi theorem by ergodic theory \cite{Fu}.

The type of weighted ergodic averages, where the weight stems from ergodic averages of another system, falls into the study of return times averages, which was initially studied by Brunel \cite{Br}. Bourgain obtained one of the key results in this field of study \cite{BoRet} (which is known as the return times theorem), while a simpler proof of his result was obtained shortly after by his joint work with Furstenberg, Katznelson, and Ornstein \cite{BFKO}. The idea of using the sequence of multiple recurrence averages as a weight can be first found in a work of the first author \cite{A}. 

\subsection{Reduction}\label{reduction}
We will reduce the theorem to a case where either one of the functions $f_1$ or $f_2$ is orthogonal to the appropriate Host-Kra-Ziegler factor (cf. \cite{HK, Z}). To do so, we consider the following criterion of weighted ergodic averages due to Frantzikinakis and Host \cite{FH}. We recall that if $G$ is an $s$-step nilpotent Lie group and $\Gamma$ is a discrete co-compact subgroup of $G$, then we call $G/\Gamma$ an $s$-step nilmanifold. A measure preserving system in which an action defined by a left translation of $G/\Gamma$ is an $s$-step nilsystem (with the normalized Haar measure).  Given $d \in \NN$, a nilsequence is a bounded sequence of the form $\psi(n_1, n_2, \ldots, n_d) := F(g_1^{n_1}g_2^{n_2}\cdots g_d^{n_d}x)$, where $g_i \in G$ for $1 \leq i \leq d$, $x \in G/\Gamma$, and $F: G/\Gamma \to \CC$ is a continuous function.

\begin{theorem}[{\cite[Theorem 2.4 and Proposition 6.1]{FH}}]\label{FH}
	Let $k \in \NN$. If $w \in \ell^\infty(\NN)$ is a sequence such that the limit
	\[ \lim_{M, N \to \infty} \frac{1}{MN} \sum_{m=1}^M \sum_{n=1}^N w(n) \overline{w(m)} \psi(n, m) \]
	exists for every $k$-step nilsequence $\psi: \NN^2 \to \CC$, then for every  measure-preserving system with commuting transformations $(Y, \mathcal{G}, \nu, S_1, \ldots, S_k)$ and functions $g_1, g_2, \ldots, g_k \in L^\infty(\nu)$, the limit
	\[ \frac{1}{N} \sum_{n=1}^N w(n) \prod_{i=1}^k g_i \circ S_i^n \]
	exists in $L^2(\nu)$-norm.
\end{theorem}

In order to utilize this criterion, we need the following convergence result due to Leibman:

\begin{theorem}[{\cite[Theorem A]{L}}]\label{Leibman}
	Let $\psi: \NN^2 \to \CC$ be a nilsequence. Then for any F{\o}lner sequence $\{ \Phi_N\}$ in $\ZZ^2$, the limit $\lim_{N \to \infty} (1/|\Phi_N|) \sum_{n \in \Phi_N} \psi(n)$ exists.
\end{theorem}

Suppose that $f_1$ and $f_2$ both belong to the $(k+1)$-th Host-Kra-Ziegler factor $\mathcal{Z}_{k+1}$. Our goal is to show that if $w(n) = f_1(T^{an}x)f_2(T^{bn}x)$, then $w(n)$ satisfies the criterion given in Theorem \ref{FH}. For simplicity we will assume that for both $i = 1, 2$, $\norm{f_i}_{L^\infty(\mu)} \leq 1$. We recall that $(X, \mathcal{Z}_{k+1}, \mu, T)$ is an inverse limit of $k+1$ nilsystems by the structure theorem of Host and Kra \cite{HK}. This result can be re-interpreted as follows: If $f_1, f_2 \in L^\infty(\mathcal{Z}_{k+1})$, then for any $\epsilon > 0$, there exists a $k+1$-th nilsystem $(M, \rho, T)$ and a continuous factor map $\pi: Z_s \to M$ such that for $i = 1, 2$, $f_i$ can be written as
	\[f_i = f_i^{nil} \circ \pi + f_i^{sml} \, , \]
	where $f_i^{nil}$ is a continuous function on $M$ for which $\norm{f_i^{nil}}_{L^\infty(\rho)} \leq \norm{f_i}_{L^\infty(\mu)}$, and $\norm{f_i^{sml}}_{L^1(\mu)} < \epsilon$.
 Let $w(n) := f_1(T^{an}x)f_2(T^{bn}x)$ and $\psi \in \ell^\infty(\NN^2)$ be a nilsequence. Then
\begin{align*} 
\frac{1}{NM} \sum_{n=1}^N \sum_{m=1}^M w(n)\overline{w(m)}\psi(n, m)
&= \frac{1}{NM}\sum_{n=1}^N f_1^{nil}(T^{an}x)f_2(T^{bn}x) \sum_{m=1}^M \overline{w(m)}\psi(n, m)  \\
&+ \frac{1}{NM}\sum_{n=1}^N f_1^{sml}(T^{an}x)f_2(T^{bn}x) \sum_{m=1}^M\overline{w(m)}\psi(n, m).
\end{align*}
We focus on the second term on the right-hand side of the equation above. We note that
\begin{align*}
&\left|\frac{1}{NM}\sum_{n=1}^N f_1^{sml}(T^{an}x)f_2(T^{bn}x) \sum_{m=1}^M\overline{w(m)}\psi(n, m)\right| \\
&\leq \frac{1}{N} \sum_{n=1}^N  \left|f_1^{sml}(T^{an}x)f_2(T^{bn}x)\right| \left| \frac{1}{M} \sum_{m=1}^M \overline{w(m)}\psi(n, m)\right|\\
&\leq \norm{\psi}_{\ell^\infty(\NN^2)} \frac{1}{N} \sum_{n=1}^N |f_1^{sml}(T^{an}x)|.
\end{align*}
Since $T$ is ergodic, recall that there exists an integral kernel
\[ K(x, y) = \ell \sum_{i=1}^\ell \ind_{A_i}(x)\ind_{A_i}(y)\]
for some $\ell \in \NN$ and $T^a$-invariant sets $\{A_i : i=1, 2, \ldots, \ell\}$ such that every $T^a$-invariant function $\xi$ can be expressed as
\[\xi(x) = \int \xi(y) K(x, y) d\mu(y) \]
(see \cite[Lemma 5.2]{ADM} for a proof). This, together with the Birkhoff-Khinchin ergodic theorem tells us that for $\mu$-a.e. $x \in X$, we have
\[\lim_{N \to \infty} \frac{1}{N} \sum_{n=1}^N |f_1^{sml}(T^{an}x)| = \int K(x, y) |f_1^{sml}|(y) d\mu(y) \leq \ell \norm{f_1^{sml}}_{L^1(\mu)} < \ell \epsilon. \]
Combining all of these together, we see that there exists a constant $C$ (that may depend on $\psi$ and $a$) for which
\[\limsup_{N, M \to \infty} \left|\frac{1}{NM}\sum_{n=1}^N f_1^{sml}(T^{an}x)f_2(T^{bn}x) \sum_{m=1}^M\overline{w(m)}\psi(n, m)\right| < C\epsilon. \]
Set $w^{nil}(n) := f_1^{nil} \circ \pi (T^{an}x) f_2^{nil} \circ \pi (T^{bn}x)$. We denote $\Re(z)$ to be the real part of the complex number $z$. Applying this decomposition trick to the other cases (for $f_2, \overline{f_1}$, and $\overline{f_2}$), we get
\begin{align*}
&\limsup_{N, M \to \infty} \Re \left( \frac{1}{NM} \sum_{n=1}^N \sum_{m=1}^M w(n)\overline{w(m)}\psi(n, m) \right) \\
&\leq \limsup_{N, M \to \infty} \Re \left( \frac{1}{NM} \sum_{n=1}^N \sum_{m=1}^M w^{nil}(n)\overline{w^{nil}(m)}\psi(n, m) \right) + C\epsilon,
\end{align*}
and
\begin{align*}
&\liminf_{N, M \to \infty} \Re \left( \frac{1}{NM} \sum_{n=1}^N \sum_{m=1}^M w(n)\overline{w(m)}\psi(n, m) \right) \\
&\geq \liminf_{N, M \to \infty} \Re \left( \frac{1}{NM} \sum_{n=1}^N \sum_{m=1}^M w^{nil}(n)\overline{w^{nil}(m)}\psi(n, m) \right) - C\epsilon,
\end{align*}
for some constant $C > 0$ that may depend on $\psi$, $a$, and $b$.

Note that $\{w^{nil}(n)\}_{n \in \NN}$ is a nilsequence, which implies that $\{w^{nil}(n)\overline{w^{nil}(m)}\psi(n, m)\}_{(n, m) \in \NN^2}$ is a nilsequence as well. Theorem $\ref{Leibman}$ implies that the limit
\[\lim_{N, M \to \infty}\frac{1}{NM} \sum_{n=1}^N \sum_{m=1}^M w^{nil}(n)\overline{w^{nil}(m)}\psi(n, m) \]
exists. Therefore,
\[\limsup_{N, M \to \infty} \Re \left( \frac{1}{NM} \sum_{n=1}^N \sum_{m=1}^M w(n)\overline{w(m)}\psi(n, m) \right) - \liminf_{N, M \to \infty} \Re \left( \frac{1}{NM} \sum_{n=1}^N \sum_{m=1}^M w(n)\overline{w(m)}\psi(n, m) \right) < 2C\epsilon, \]
and since our choice of $\epsilon$ is arbitrary, we can conclude that the limit
\[\lim_{N, M \to \infty} \Re \left( \frac{1}{NM} \sum_{n=1}^N \sum_{m=1}^M w(n)\overline{w(m)}\psi(n, m) \right) \]
exists, and so does its imaginary counterpart. Therefore, we know that the limit
\[ \lim_{N, M \to \infty} \frac{1}{NM} \sum_{n=1}^N \sum_{m=1}^M w(n)\overline{w(m)}\psi(n, m) \]
exists, so by Theorem \ref{FH}, Theorem \ref{mainThm} holds for the case where $f_1$ and $f_2$ belong to the $k+1$-th Host-Kra-Ziegler factor for $\mu$-a.e. $x \in X$; we call this set of full measure $\hat{X}$.

It remains to show the existence of a set of full measure in which the averages converge in a case either $f_1$ or $f_2$ belongs to the orthogonal complement of $L^2(\mathcal{Z}_k)$. In other words, we need to prove the following:

\begin{theorem}\label{CVto0Thm} Let the notations be as in Theorem $\ref{mainThm}$. Suppose that $T$ is ergodic. If either $f_1$ or $f_2$ belongs to the orthogonal complement of the $k+1$-th Host-Kra-Ziegler factor of $T$, then there exists a set of full measure $\tilde{X}$ such that for any $x \in \hat{X}$, we have
	\begin{equation}\label{CVto0}
	\limsup_{N \to \infty} \norm{\frac{1}{N} \sum_{n=0}^{N-1} f_1(T^{an}x)f_2(T^{bn}x)\prod_{i=1}^k g_i \circ S_i^k}_{L^2(\nu)} = 0.
	\end{equation}
\end{theorem}
We take $X' := \hat{X} \cap \tilde{X}$, and this is the desired set of full measure to conclude the proof of Theorem \ref{mainThm}.

\subsection{Remark}
A purpose of this update to our preprint from 2015 is to fill in a gap in the proof that exists in our previous version (this gap has been filled by the argument made in \S\ref{reduction}). In the previous version, we have attempted to utilize the box seminorms that was introduced by Host \cite{H} as well as some machinery of higher order Fourier analysis. It should be noted that Theorem \ref{FH} had not been announced when the previous version of the preprint was being prepared (the last update of this preprint was announced on 21 September 2015 on arXiv, while the Frantzikinakis-Host paper was first announced on 18 November 2015). 

This note is a sequel to \cite{AM}, which proves Theorem \ref{mainThm} for the case where each $S_i$ is a power of a single measure-preserving transformation (i.e. $S_i = S^i$ for $i=1, 2, \ldots, k$ for some measure-preserving transformation $S$ on $(Y, \mathcal{G}, \nu)$). That manuscript was accepted for publication on 26 July 2015, and the result was proven without using Theorem \ref{FH}.

The argument made in the remainder of the paper has existed since the previous versions of the preprint, and no significant change has been made in this update. Another main purpose of this note is to show that the Wiener-Wintner type result, such as the one obtained in \cite{ADM}, can be used to show convergence of more complex return times type averages.

\section{Proof of Theorem \ref{CVto0Thm}}
For sake of simplicity, we assume that the functions $f_1, f_2, g_1, g_2, \ldots, g_k$ are real-valued, and that their respective $L^\infty$-norms are bounded above by $1$.

The proof presented here is analogous to that of the proof of \cite[Theorem 1.5(a)]{AM} for the case we had a single measure-preserving transformation $S$ (i.e. $S_i = S^i$). We recall the following inequality that was obtained in the proof of the double recurrence Wiener-Wintner result \cite{ADM}:
\begin{equation}\label{estimate}
\int \limsup_{N \to \infty} \sup_{t \in \RR} \left|\frac{1}{N} \sum_{n=0}^{N-1} f_1(T^{an}x)f_2(T^{bn}x) e^{2\pi i n t} \right|^2 d\mu(x) \lesssim_{a, b} \min_{i = 1, 2}\vertiii{f_i}_3^2.
\end{equation}
In this section, we will denote $a_1 = a$ and $a_2 = b$. Furthermore, for every $l \in \NN$ and for every $\vec{h}(l) := (h_1, h_2, \ldots, h_l) \in \NN^l$, we denote
\[ \begin{array}{ll}
F_{1, \vec{h}(1)} = f_1 \cdot f_1 \circ T^{a_1h_1},
& F_{2, \vec{h}(1)} = f_2 \cdot f_2 \circ T^{a_2h_1}, \\
F_{1, \vec{h}(2)} = F_{1, \vec{h}(1)} \cdot F_{1, \vec{h}(1)} \circ T^{a_1h_2},
& F_{2, \vec{h}(2)} = F_{2, \vec{h}(1)} \cdot F_{2, \vec{h}(1)} \circ T^{a_2h_2}, \\
\cdots, & \cdots, \\
F_{1, \vec{h}(k-1)} = F_{1, \vec{h}(k-2)} \cdot F_{1, \vec{h}(k-2)} \circ T^{a_1h_{k-1}},
& F_{2, \vec{h}(k-1)} = F_{2, \vec{h}(k-2)} \cdot F_{2, \vec{h}(k-2)}\circ T^{a_2h_{k-1}} \, .
\end{array} \]

\begin{lemma}\label{dynEstMultLemma}
	Let all the notations be as above. Then for each positive integer $k \geq 2$, we have
	\begin{align}\label{dynEstMult}
	&\limsup_{N \to \infty} \norm{\frac{1}{N} \sum_{n=1}^N f_1(T^{a_1n}x)f_2(T^{a_2n}x) \prod_{i=1}^k g_i \circ S_i^{n}}_{L^2(\nu)}^2 \\
	&\lesssim_{a_1, a_2} \liminf_{H_1 \to \infty}  \left(\frac{1}{H_1} \sum_{h_1=1}^{H_1} \liminf_{H_2 \to \infty} \frac{1}{H_2} \sum_{h_2=1}^{H_2} \cdots \right. \nonumber \\
	& \left. \liminf_{H_{k-1} \to \infty} \frac{1}{H_{k-1}} \sum_{h_{k-1}=1}^{H_{k-1}} \limsup_{N \to \infty} \sup_{t \in \RR} \left| \frac{1}{N} \sum_{n=1}^N F_{1, \vec{h}(1)}(T^{a_1n}x)F_{2, \vec{h}(1)}(T^{a_2n}x) e^{2\pi i n t} \right|^2 \right)^{2^{-(k-1)}} \, . \nonumber
	\end{align}
\end{lemma}
\begin{proof}
	We will show this by induction on $k$. The prove the base case $k = 2$, we first apply van der Corput's lemma to see that
	\begin{align*} &\limsup_{N \to \infty} \norm{\frac{1}{N} \sum_{n=0}^{N-1} f_1(T^{a_1n}x)f_2(T^{a_2n}x)g_1(S_1^ny)g_2(S_2^ny)}_{L^2(\nu)}^2 \\
	&\lesssim \liminf_{H_1 \to \infty} \frac{1}{H_1} \sum_{h_1=0}^{H_1-1}\limsup_{N \to \infty}  \int \left|(g_1 \cdot g_1 \circ S_1^h)(y)  \frac{1}{N} \sum_{n=0}^{N-1} F_{1, \vec{h}(1)}(T^{an}x)F_{2, \vec{h}(1)}(T^{bn}x) (g_2 \cdot g_2 \circ S_2^h)((S_2S_1^{-1})^ny)\right| d\nu. \end{align*}
	By H\"{o}lder's inequality (and recalling that $\norm{g_1}_{L^\infty(\nu)} \leq 1$), we dominate the last line above by
	\[\liminf_{H_1 \to \infty} \frac{1}{H_1} \sum_{h_1=0}^{H_1-1}  \left(\limsup_{N \to \infty}  \int 
	\left|\frac{1}{N} \sum_{n=0}^{N-1} F_{1, \vec{h}(1)}(T^{a_1n}x)F_{2, 
	\vec{h}(1)}(T^{a_2n}x) (g_2 \cdot g_2 \circ S_2^h)((S_2S_1^{-1})^ny)\right|^2 d\nu\right)^{1/2}. \]
	Let $\sigma_{g \cdot g \circ S_2^h}$ be the spectral measure of $\TT$ for the function $g \cdot g \circ S_2^h$ for each $h$, with respect to the transformation $S_2S_1^{-1}$ (this is a Borel probability measure on $\TT$). By the spectral theorem, the last expression becomes
	\[\liminf_{H_1 \to \infty} \frac{1}{H_1} \sum_{h_1=0}^{H_1-1}  \left(\limsup_{N \to \infty} \int  \left| \frac{1}{N} \sum_{n=0}^{N-1} F_{1, \vec{h}(1)}(T^{a_1n}x)F_{2, \vec{h}(1)}(T^{a_2n}x)e^{2\pi int} \right|^2 d\sigma_{g_2 \cdot g_2 \circ S_2^h}(t)\right)^{1/2}, \]
	which is bounded above by
	\[\liminf_{H_1 \to \infty} \frac{1}{H_1} \sum_{h_1=0}^{H_1-1}\left( \limsup_{N \to \infty}\sup_{t \in \RR} \left| \frac{1}{N} \sum_{n=0}^{N-1} F_{1, h_1}(T^{a_1n}x)F_{2, \vec{h}(1)}(T^{a_2n}x)e^{2\pi int} \right|^2\right)^{1/2}. \]
	After we apply the Cauchy-Schwarz inequality (on the averages over $H_1$), we obtain the desired inequality for the case $k = 2$.
	
	Now suppose the estimate holds when we have $k-1$ terms. By applying van der Corput's lemma and the Cauchy-Schwarz inequality, the left hand side of the estimate $(\ref{dynEstMult})$ is bounded above by the product of a constant that only depends on the values of $a_1$ and $a_2$ and 
	\[ \liminf_{H_1 \to \infty} \left( \frac{1}{H_1} \sum_{h_1 = 1}^{H_1} \limsup_{N \to \infty} \norm{\frac{1}{N} \sum_{n=1}^N F_{1, \vec{h}(k-1)} (T^{a_1n}x)F_{2, \vec{h}(k-1)} (T^{a_2n}x) \prod_{i=2}^k (g_i \cdot g_i \circ S_i^{h_1}) \circ (S_iS_1^{-1})^n}_{L^2(\nu)}^2 \right)^{1/2} \, , \] 
	and we can apply the inductive hypothesis on this $\limsup$ of the square of the $L^2$-norm above and the Cauchy-Schwarz inequality to obtain the desired estimate.
\end{proof}
The preceding lemma allows us to identify the desired set of full measure for each positive integer $k$.
\begin{proof}[Proof of Theorem $\ref{CVto0Thm}$]
	We will first show that for each positive integer $k \geq 1$, there exists a set of full measure $\tilde{X}_k$ such that the statement of Theorem $\ref{CVto0Thm}$ holds for this particular $k$.
	
	The set $\tilde{X}_1$ can be obtained from the double recurrence Wiener-Wintner result \cite{ADM} by applying the spectral theorem. For $k \geq 2$, we consider a set
	\begin{align*}&\tilde{X}_k = \left\{x \in X: \liminf_{H_1 \to \infty} \left( \frac{1}{H_1} \sum_{h_1=1}^{H_1} \liminf_{H_2 \to \infty} \frac{1}{H_2} \sum_{h_2=1}^{H_2} \cdots \right. \right. \\
	&\left. \left. \liminf_{H_k \to \infty}\frac{1}{H_{k-1}} \sum_{h_k=1}^{H_k} \limsup_{N \to \infty} \sup_{t \in \RR} \left| \frac{1}{N} \sum_{n=1}^N F_{1, \vec{h}(k-1)}(T^{a_1n}x)F_{2, \vec{h}(k-1)}(T^{a_2n}x) e^{2\pi int}\right|^2 \right)^{2^{-(k-1)}} = 0 \right\}\, . \end{align*}
	We will show that the set on the right hand side is indeed the desired set of full measure. To first show that $\mu(\tilde{X}_k) = 1$, we  compute that 
	\begin{align}
	&\label{inductInt}\int \liminf_{H_1 \to \infty} \left(\frac{1}{H_1} \sum_{h_1=1}^{H_1} \liminf_{H_2 \to \infty} \frac{1}{H_2} \sum_{h_2=1}^{H_2} \cdots \right. \\
	&\left. \liminf_{H_{k-1} \to \infty}\frac{1}{H_{k-1}} \sum_{h_{k-1}=1}^{H_{k-1}} \limsup_{N \to \infty} \sup_{t \in \RR} \left| \frac{1}{N} \sum_{n=1}^N F_{1, \vec{h}(k-1)}(T^{a_1n}x)F_{2, \vec{h}(k-1)}(T^{a_2n}x)e^{2\pi int} \right|^2 \right)^{2^{-(k-1)}} \, d\mu = 0, \nonumber
	\end{align}
	which would show that the non-negative term inside the integral equals zero for $\mu$-a.e. $x \in X$. To do so, we apply Fatou's lemma and H\"{o}lder's inequality to show that the integral above is bounded above by
	\begin{align*}
	& \liminf_{H_1 \to \infty} \left(\frac{1}{H_1} \sum_{h_1=1}^{H_1}  \liminf_{H_2 \to \infty} \frac{1}{H_2} \sum_{h_2=1}^{H_2} \cdots \right. \\
	&\left. \liminf_{H_{k-1} \to \infty}\frac{1}{H_{k-1}} \sum_{h_{k-1}=1}^{H_{k-1}} \int \limsup_{N \to \infty} \sup_{t \in \RR} \left| \frac{1}{N} \sum_{n=1}^N F_{1, \vec{h}(k-1)}(T^{a_1n}x)F_{2, \vec{h}(k-1)}(T^{a_2n}x)e^{2\pi int} \right|^2 \, d\mu \right)^{2^{-(k-1)}}.
	\end{align*}
	Note that the last integral is bounded above by $\displaystyle{C \cdot \min_{i=1, 2} \vertiii{F_{i, \vec{h}(k-1)}}_3^2}$ by the estimate $(\ref{estimate})$, where $C$ is a constant that only depends on $a_1$ and $a_2$. By letting $H_j$ go to infinity for each $j = 1, 2, \ldots, k-1$, we conclude that the integral on the left hand side of $(\ref{inductInt})$ is bounded above by $C$ times the minimum of the power of $\vertiii{f_1}_{k+2}$ or $\vertiii{f_2}_{k+2}$. Since either $f_1$ or $f_2$ belongs to $\mathcal{Z}_{k+1}(T)^\perp$, we know that either $\vertiii{f_1}_{k+2} = 0$ or $\vertiii{f_2}_{k+2} = 0$. Thus, $(\ref{inductInt})$ holds, which implies that $\tilde{X}_k$ is indeed a set of full measure.
	
	Now we need to show that if $x \in \tilde{X}_k$, then $(\ref{CVto0})$ holds. But this follows immediately from Lemma $\ref{dynEstMultLemma}$, since if $x \in \tilde{X}_k$, the right hand side of $(\ref{dynEstMult})$, which is an upper bound for the $\limsup$ of the averages in $(\ref{CVto0})$, is $0$. 
	
	Hence, we conclude the proof by setting $\tilde{X} = \bigcap_{k=1}^\infty \tilde{X}_k$. We note that $\tilde{X}$ is a countable intersection of sets of full measures, so $\tilde{X}$ must be a set of full measure as well. 
\end{proof}

\section*{Acknowledgment}
We thank the anonymous referee for informing us about a gap in the argument in our previous preprint.

\end{document}